\documentclass[11pt,leqno]{amsart}

\usepackage[T1]{fontenc}
\usepackage{amsthm,amsfonts,amssymb,amsmath,oldgerm,mathrsfs,mathabx,enumitem}
\usepackage[scr=boondoxo]{mathalfa}
\numberwithin{equation}{section}
\usepackage{fullpage,setspace,fancyhdr,bm}
\usepackage{graphicx,psfrag,listings} 
\usepackage{color}

\usepackage[top=30mm,bottom=30mm,left=25mm,right=25mm,a4paper]{geometry}

\usepackage{hyperref}
\usepackage{graphics}

\usepackage{float}

\setlist[itemize,1]{label=\ensuremath{\diamond}}

\theoremstyle{plain}
\newtheorem{theorem}{Theorem}

\newtheorem{proposition}[theorem]{Proposition}
\newtheorem{lemma}[theorem]{Lemma}

\theoremstyle{definition}

\newtheorem{assumption}[theorem]{Assumption}

\newcommand{\kpp}{\text{kpp}}
\newcommand{\profile}{p}

\newcommand{\deq}{:=}
\newcommand{\rdeq}{=:}
\newcommand{\re}{\mathrm{e}}

\DeclareMathOperator{\Real}{Re}
\DeclareMathOperator{\Imag}{Im}

\newcommand{\norm}[1]{\lVert#1\rVert}
\newcommand{\Norm}[1]{\left\lVert#1\right\rVert}
\newcommand{\absolute}[1]{\lvert#1\rvert}
\newcommand{\Absolute}[1]{\left\lvert#1\right\rvert}
\newcommand{\scalp}[1]{\langle #1\rangle} 
\newcommand{\Scalp}[1]{\left\langle #1\right\rangle}

\newcommand{\transp}[1]{#1^\mathrm{T}}
\newcommand{\Set}[1]{\left\lbrace#1\right\rbrace}

\newcommand{\CC}{{\mathbb C}}

\newcommand{\RR}{{\mathbb R}}

\newcommand\uu{{\underline u}}

\newcommand\cC{{\mathcal C}}

\newcommand\cL{{\mathcal L}}

\newcommand\cO{{\mathcal O}}

\newcommand\cQ{{\mathcal Q}}
\newcommand\cR{{\mathcal R}}

\title{Linear convective stability of a front superposition with unstable connecting state}

\author{Louis Garénaux}
\address{Karlsruhe Institute for Technology, Englerstraße 2, 76131 Karlsruhe, Germany}
\email{{\tt louis.garenaux@kit.edu}}
\thanks{Research of LG was Funded by the Deutsche Forschungsgemeinschaft (DFG, German Research Foundation) – Project-ID 258734477 – SFB 1173.}

\author{Bastian Hilder}
\address{Department of Mathematics, Technical University of Munich, Boltzmannstraße 3, 85748 Garching, Germany}
\email{{\tt bastian.hilder@tum.de}}
\thanks{BH~was partially supported by the Swedish Research Council -- grant no.~2020-00440 -- and the Deutsche Forschungsgemeinschaft (DFG, German Research Foundation) -- Project-IDs 444753754 and 543917644.}

\begin{document}

\begin{abstract}
We study convective stability of a two-front superposition in a reaction-diffusion system. Due to the instability of the connecting equilibrium, long-range semi-strong interaction is expected between the two waves. When restricting to the linear dynamic, we indeed identify that convective stability of superposed waves occurs for fewer propagation speeds than for the corresponding single waves. It reflects the interaction that monostable waves have over long distances. Our method relies on numerical range estimates, that imply time-uniform resolvent bounds. 

\vspace{0.5em}


\vspace{0.5em}

\end{abstract}

\maketitle

\tableofcontents

\section{Introduction}

\subsection{Setting and main result}
We study a reaction-diffusion system that models interaction between two invasive species:
\begin{equation}
\label{e:rd}
\begin{cases}
\partial_t u_1 = d\partial_{xx} u_1 + r u_1(1 - u_1) + \alpha_1 u_1 u_2\\
\partial_t u_2 = \partial_{xx} u_2 + u_2(1 - u_2) + \alpha_2 u_1 u_2
\end{cases}
\hspace{4em}
t>0, \quad x\in \RR,
\end{equation}
with positive parameters $(d, r, \alpha_1, \alpha_2)$. 
With $D = \mathrm{diag}(d,1)$ and 
\begin{equation*}
g(u) = 
\begin{pmatrix}
r u_1 (1 - u_1) + \alpha_1 u_1 u_2\\
u_2 (1 - u_2) + \alpha_2 u_1 u_2
\end{pmatrix},
\end{equation*}
system \eqref{e:rd} conveniently rewrites as 
\begin{equation*}
u_t = D u_{xx} + g(u).
\end{equation*}
This system admits four constant equilibria
\begin{equation}
\label{e:constant-solutions}
e_1 = 
\begin{pmatrix}
\frac{r+\alpha_1}{r - \alpha_1 \alpha_2} \\[1ex] \frac{r(1+\alpha_2)}{r - \alpha_1 \alpha_2}
\end{pmatrix},
\hspace{2em}
e_2 =
\begin{pmatrix}
0 \\ 1
\end{pmatrix},
\hspace{2em} 
e_3 =
\begin{pmatrix}
1 \\ 0
\end{pmatrix},
\hspace{2em} 
e_4 = 
\begin{pmatrix}
0 \\ 0
\end{pmatrix},
\end{equation}
that correspond (respectively) to the cohabitation of the two species, the species 2 only, the species 1 only, and an empty environment. We refer to \cite{Iida-Lui-Ninomiya-11} for similar systems involving more equilibria. We assume in the following that \eqref{e:rd} is cooperative, and that the first species has a higher invasion speed than the second one.
\begin{assumption}
\label{a:parameters}
Parameters satisfy $d>1$, $r>1$, and the stability condition 
\begin{equation*}
r - \alpha_1 \alpha_2 > 0.
\end{equation*}
In particular, $e_1$ is spectrally stable while $e_2$, $e_3$ and $e_4$ are spectrally unstable.
\end{assumption}

In addition to constant solutions, \eqref{e:rd} also admits families of invasion front solutions, as discussed in Section \ref{s:single-front}. These solutions have a fix profile $\profile \in \cC^\infty(\RR,\RR^2)$ converging at $\pm \infty$ to equilibria \eqref{e:constant-solutions}, and propagate at constant speed $c\in \RR$:
\begin{equation}
\label{e:traveling-wave}
u(t,x) = \profile(x - c t).
\end{equation}
We are interested in two-stage invasions, which are superpositions of two invasion fronts $(\profile_1, c_1)$ and $(\profile_2, c_2)$ that respectively connect $e_1 \to e_3$ and $e_3 \to e_4$. While both waves are well described separately, the combined pattern $e_1 \to e_3 \to e_4$ is not fully understood when $c_1 < c_2$. Indeed, due to the different front velocities, it is time-dependent in any frame, and thus cannot be constructed as the solution to a time-independent ODE. It is common to rely on a stability approach to describe such a time-dependent pattern, which  comes down to understand the interactions between the two waves.

In most settings, two waves with different propagation speeds interact weakly enough that they behave as if studied separately. A more detailed literature discussion can be found in subsection \ref{s:long-range-interaction}. In contrast, the last works \cite{Holzer-Scheel-14,Faye-Holzer-19a,Girardin-Lam-19} on systems related to \eqref{e:rd} indicate that two invasion fronts interact over large distances. In particular, the propagation speed of one front can be affected by the presence of an other front, even though their distance grows linearly in time. See again subsection \ref{s:long-range-interaction}. 

To contribute to the description of invasion front superposition, we follow a stability approach. We highlight conditions on $(c_1, c_2)$ ensuring stability of the superposed pattern. Since monostable fronts usually come in speed-parameterized families, we focus on identification of stable speed pairs $(c_1, c_2)$.

Let $\chi : \RR \to [0,1]$ be a $\cC^\infty$, monotone partition of $\RR$:
\begin{equation*}
\chi(x) = 
\begin{cases}
0 & x \leq -1, \\
1 & x \geq 1.
\end{cases}
\end{equation*}
Let $\psi_1 < 0 < \psi_2$ be two initial positions, and let $c_0 = \frac{c_1 + c_2}{2}$. The resulting front superposition is defined as 
\begin{equation}
\label{e:terrace-ansatz}
\uu(t,x) =
\big(1 - \chi(x - c_0 t)\big)
\profile_1(x - c_1 t - \psi_1) + \chi(x - c_0 t) \profile_2(x - c_2 t - \psi_2).
\end{equation}

\begin{assumption}
\label{a:ordered-fronts}
The two fronts remain far apart: Speeds and positions are ordered as $c_1 < c_2$ and $\psi_2 - \psi_1 \gg 0$.
\end{assumption}

Although $\uu$ is probably not a solution to \eqref{e:rd}, the stability approach for wave superposition is to look for solutions that remain close to $\uu$. See \cite{Wright-09} for a successful application of this approach. How good of an approximated solution $\uu$ is will be discussed in section \ref{s:residual}. In the present contribution, we focus on the linearization of \eqref{e:rd} at $\uu$:
\begin{equation}
\label{e:main}
v_t = D v_{xx} + J_g(\uu(t,x)) v,
\end{equation}
where $J_g$ denotes the Jacobian matrix of $g$. The instability of $e_3$ and $e_4$ is expected to impact the long-time dynamics of \eqref{e:main}. When studying a single monostable front, the invaded state instability has a well-understood impact. In particular, the stability of the wave in a co-moving frame can be recovered using spatial exponential weights \cite{Sattinger-76}. This mechanism is known as convective stability \cite{Ghazaryan-Latushkin-Schecter-13}.

For a wave superposition, the relation between equilibria instability and full pattern stability through spatial localization has not been described yet. We investigate this connection by restricting to bounded weights. Our conclusion is that convective stability of individual equilibria is not enough for the wave to be stable.
\begin{assumption}
\label{a:speeds}
The low speed $c_1$ is such that there exists a $\kappa_1$ satisfying
\begin{equation}
\label{e:localization-rate-1}
\begin{cases}
{\kappa_1}^2 - c_1 \kappa_1 + (1+\alpha_2) & < 0,\\
d{\kappa_1}^2 - c_1 \kappa_1 - r & < 0.
\end{cases}
\end{equation}
The high speed $c_2$ is such that there exists a $\kappa_2$ satisfying
\begin{equation}
\label{e:localization-rate-2}
d {\kappa_2}^2 - c_2 \kappa_2 + r < 0.
\end{equation}
Furthermore, $c_1$, $c_2$, $\kappa_1$ and $\kappa_2$ satisfy
\begin{equation}
\label{e:localization-rate-3}
d{\kappa_2}^2 - c_2 \kappa_2 + r + \kappa_1 (c_2 - c_1) < 0.
\end{equation}
\end{assumption}

Conditions \eqref{e:localization-rate-1} and \eqref{e:localization-rate-2} are equivalent to $c_1$ convective stability of $e_3$ and $c_2$ convective stability of $e_4$ (respectively). The additional condition \eqref{e:localization-rate-3} is more restrictive than \eqref{e:localization-rate-2}, and thus rules out some speed pairs $(c_1, c_2)$. As noticed in Lemma \ref{l:no-remnant-instability}, it is always possible to fulfill \eqref{e:localization-rate-3} by choosing a large enough $c_2$. Thus, we understand the extra condition \eqref{e:localization-rate-3} as reflecting the two front interactions. Our main result states as follows.
\begin{theorem}
\label{t:main}
Assume that Assumptions \ref{a:parameters}, \ref{a:ordered-fronts} and \ref{a:speeds} are satisfied. There exist positive constants $\alpha$, $C$, $\eta$ such that the following holds. For all $0 < \alpha_1, \alpha_2 < \alpha$, there exists a $\cC^\infty$ weight $\omega : (0,+\infty) \times \RR \to (0,1]$, such that solutions to \eqref{e:main} with initial data $v_0$ satisfy
\begin{equation*}
\Norm{\frac{v(t)}{\omega(t)}}_{L^2(\RR)} \leq C \re^{-\eta t} \Norm{\frac{v_0}{\omega(0)}}_{L^2(\RR)}.
\end{equation*}
\end{theorem}

\subsection{Long range semi-strong interactions}
\label{s:long-range-interaction}
A two wave interaction that modifies the wave shapes is qualified as semi-strong, see \cite{vanHeijster-Doelman-Kaper-Promislow-10} and the references therein. In contrast, weakly interacting waves behave as if studied separately, while strong interaction refers to wave collision.

Other reaction-diffusion waves whose distance grows linearly have weak interactions. See \cite{Wright-09,Fife-McLeod-77,Lin-Schecter-15} for pulses and bistable fronts superposition without speed modification or position drifts.\footnote{We believe that \cite{Wright-09} system approach adapts to bistable front superposition with very few modifications.} Waves with identical speeds interact semi-strongly, and thus present richer interactions \cite{vanHeijster-Doelman-Kaper-Promislow-10}, including unbounded position drifts. 

The situation is similar in non-parabolic equations, while constant equilibria are rarely exponentially stable in such setting. For Korteweg-de Vries, Schrödinger, wave, or Klein-Gordon equations, pulses traveling with distinct speeds interact weakly and pulses with equal speeds interact semi-strongly. See for example \cite{Martel-Merle-Tsai-02,Martel-Merle-06,Martel-Merle-16,Cote-Munoz-14} and \cite{Eychenne-Valet-23,Nguyen-19}.

In contrast to the above studies, invasion fronts interact semi-strongly even though their distance is growing linearly with time \cite{Girardin-Lam-19}. We believe that this phenomenon is related to the instability of the connecting state ($e_3$ in the present setting). Indeed, the instability of other states does not seem to strengthen interactions, see \cite{Lin-Schecter-15,Carrere-18}. 
Furthermore, instability of the intermediate state is only repaired in weighted topology. When using optimal weighted norms, generic fronts lose all their spatial localization. Thus, the interaction of two generic waves connected by an unstable state can be thought of as the interaction of two $L^\infty$ waves with no spatial localization.

\subsection{Single front dynamics}
Let us stress that the front $(p_2,c_2)$ is unstable when distance is measured with bounded weights. Indeed, behind the invasion, transport is directed towards the unstable infinity. In fact, numerical simulations suggest that initial data close to $p_2$ may converge to a front superposition similar to the one we study, see Figure \ref{fig:numerics1}.

System \eqref{e:rd} also admits an $e_1 \to e_4$ invasion front family $(p_3, c_3)_{c_3\geq c_{3,*}}$. Interestingly enough, numerical simulations also suggest that nearby initial data quickly breaks into a two-stage invasion. For different parameter values than the one considered here, numerical simulations rather showcase locked fronts, where the $p_1$-$p_2$ connection reduces to a $p_3$ wave. To relate $p_3$ convective stability to the existence of stable speed pairs $(c_1, c_2)$ satisfying $c_1 < c_2$ seems an important question to solve, see Figure \ref{fig:numerics2}.

\begin{figure}
\centering
\includegraphics[width=0.6\linewidth]{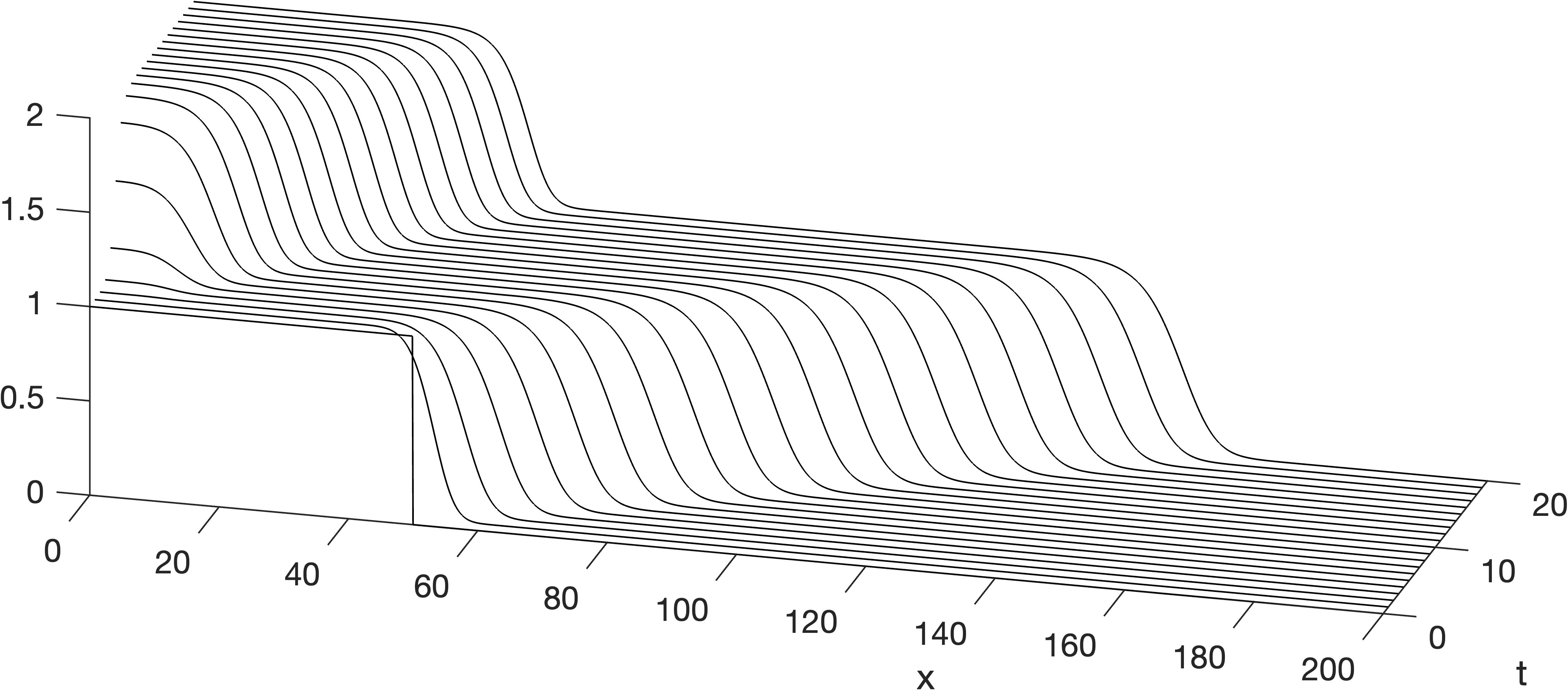}
\hspace{1cm}
\includegraphics[width=0.25\textwidth]{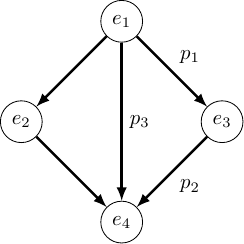}
\caption{Left: Numerical simulation showcasing $p_2$ instability. Parameter values are $(d,r,\alpha_1,\alpha_2) = (4,2,0.75,0.75)$, and the initial datum is a step function of $e_3$ for $x \in (-50,50)$ and $e_4$ everywhere else, with an additional small perturbation in $u_2$ in a neighborhood of $x = 0$. Initially the solution forms a front with a profile close to $p_2$. However, the perturbation in the wake then grows into a secondary slower front with a profile close to $p_1$. Right: All known single fronts (edges) between pairs of equilibria (nodes).}
\label{fig:numerics1}
\end{figure}

\begin{figure}
\centering
\hspace{\stretch{1}}\includegraphics[width=0.45\linewidth]{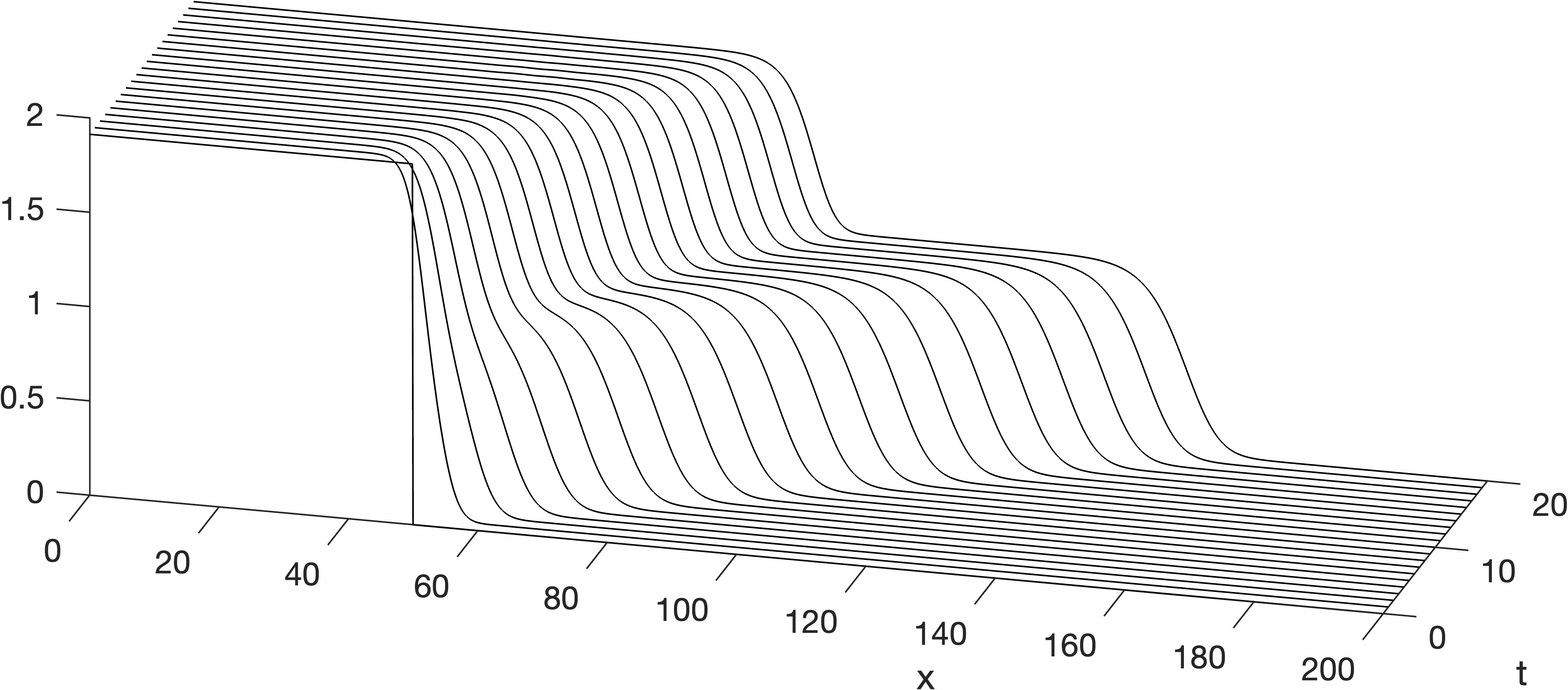}
\hspace{\stretch{1}}
\includegraphics[width=0.45\linewidth]{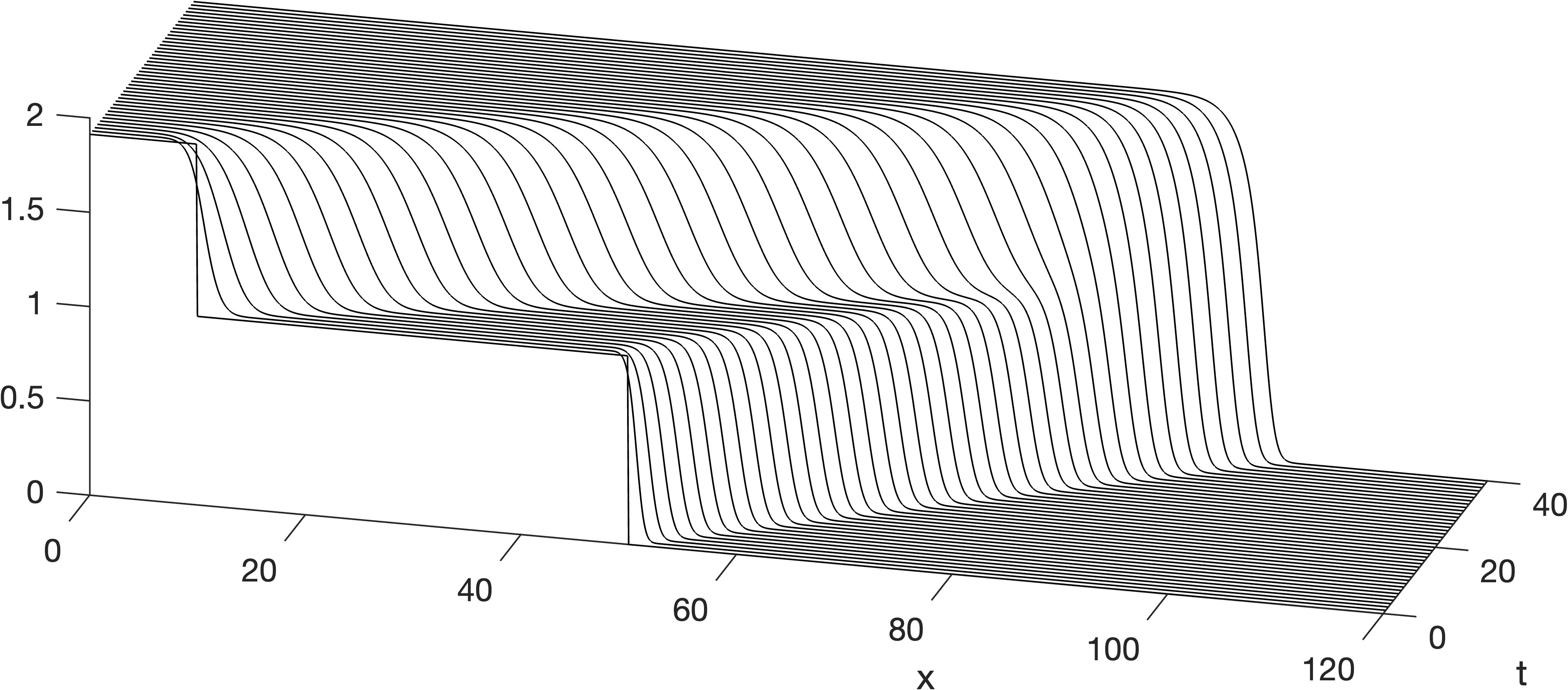}
\hspace{\stretch{1}}
\caption{Numerical simulations showcasing $p_3$ (in)stability. Left: same parameters as in Figure \ref{fig:numerics1} and initial step function with $e_1$ for $x \in (-50,50)$ and $e_4$ everywhere else. The solution initially forms a front connecting $e_1$ and $e_4$ directly, which breaks up quickly into a superposition of $p_1$ and $p_2$. Right: parameter values are $(d,r,\alpha_1,\alpha_2) = (0.2,2,0.75,0.75)$, and initial step function with $e_1$ for $x \in (-10,10)$, $e_3$ for $x \in (-40,40)\setminus (-10,10)$ and $e_4$ everywhere else. Initially a superposition front of $p_1$ and $p_2$ forms with $c_1 > c_2$. This front then collapses to a single interface $p_3$.}
\label{fig:numerics2}
\end{figure}

\subsection{Obstacles regarding the original system}
\label{s:residual}
As of now, only linear dynamics can be handled. With very few efforts, we could incorporate quadratic terms in the argument. Indeed $\omega$ is bounded and the obtained temporal decay is integrable. Thus a Duhamel argument readily covers higher order terms. The main obstacle we are facing to apply the wave separation approach from \cite{Wright-09} are residual terms. 
\begin{lemma}
Assume that the solution to \eqref{e:rd} decomposes as 
\begin{equation*}
u(t,x) = \uu(t, x) + v(t, x).
\end{equation*}
Then the correction $v$ satisfies
\begin{equation*}
v_t = \cR(\uu) + D v_{xx} + J_g(\uu) v + \cQ(v),
\end{equation*}
where residual and quadratic terms are defined as
\begin{align*}
\cR(\uu) \deq {}& - \uu_t + D \uu_{xx} + g(\uu),\\[1ex]
\cQ(v) \deq {}& g(\uu + v) - g(\uu) - J_g(\uu) v.
\end{align*}
\end{lemma}
\begin{proof}
Getting the claimed expression is direct.
\end{proof}

Since temporal decay of linear dynamics is only possible in weighted spaces, smallness of the residual must be considered in these spaces. There, it behaves to leading order as
\begin{equation}
\label{e:residual-1}
\frac{1}{\omega(t,x)} \cR(\uu(t,x)) \sim \frac{\chi'(x-c_0 t)}{\omega(t,x)} \bigg(p_2(x - c_2 t) - p_1(x - c_1 t)\bigg),
\end{equation}
such term arising from the commutator $[\partial_{xx}, \chi](p_2 - p_1)$. Although \eqref{e:residual-1} has compact support, its $L^\infty$ norm exponentially grows in time. Indeed, the spectral gap \eqref{e:localization-rate-1} requires the weight to be more localized than $p_1 - e_3$.

Although our linear analysis would apply to a different superposition ansatz, to our knowledge no choice leads to better residual behavior than \eqref{e:terrace-ansatz}. A linear superposition
\begin{equation*}
\uu(t,x) \deq p_1(x - c_1 t) + p_2(x - c_2 t) - e_3,
\end{equation*}
does not create linear terms as \eqref{e:residual-1}, since the weighted residual behaves at leading order as 
\begin{equation}
\label{e:residual-2}
\frac{1}{\omega(t,x)} B\bigg(p_2(x - c_2 t) - e_3, p_1(x - c_1 t) - e_3\bigg)
\end{equation}
for some symmetric bilinear map $B:\RR^2 \times \RR^2 \to \RR^2$. However, the absence of cut-off function $\chi$ allows for much more communication between the two profiles. In the region $x \geq c_2 t$, \eqref{e:residual-2} reduces to $\frac{p_1 - e_3}{\omega}$, and is thus insensitive to $p_2$. At $x= c_2 t$, it exponentially grows with time because of the spectral gap condition \eqref{e:localization-rate-1}.

\subsection{Future directions}
\label{s:future}
\subsubsection{Incorporating residual terms} In view of the previous subsection, the more promising direction seems to saturate the spectral gap conditions, that is to work with critical weights. In such a setting, modulation of the front positions and speeds to account for the $L^\infty$ residual as a forcing term seems an interesting scenario. A first step in this regard is to better understand the orbital stability of single invasion fronts \cite{Garenaux-Rodrigues-25}.

\subsubsection{Periodic equilibrium} In many biological and physical models, equilibria are space periodic. For example, the KPP equation with non-local interactions presents two-stage invasions with periodic equilibrium at the back. Such patterns can be described using the approach in \cite{Garenaux-24} when the two invasive equilibria are close. Looking for their stability would be an interesting direction.

\subsection{Technical summary and outline}\label{s:technical-summary}
With the change of variable $v = \omega w$, equation \eqref{e:main} equivalently rewrites as
\begin{align}
\label{e:weighted-dynamic}
w_t = {}& \cL(t) w,\\
\nonumber \deq {}& D w_{xx} + 2D\frac{\omega_x}{\omega} w_x + \left(J_g(\uu) - \frac{\omega_t}{\omega} + D \frac{\omega_{xx}}{\omega}\right) w,
\end{align}
while Theorem \ref{t:main} states temporal decay for $w$. To control the long-time dynamics of the linear equation $\eqref{e:weighted-dynamic}$, we rely on the evolution system theory \cite[chap 5]{Pazy-83}. It requires spectral stability of the parabolic operator family $\cL(t) : H^2(\RR) \subset L^2(\RR) \to L^2(\RR)$.

To obtain time-independent resolvent bounds on $\cL(t)$, we estimate its numerical range. This approach allows us to handle operators and weights that are space and time dependent. In addition, numerical range study appears convenient for scalar KPP equations \cite{Kolmogoroff-Pretrovsky-Piscounoff-37}, since it sharply controls their spectral gaps.
However, numerical range estimates poorly handle systems with coupling coefficients. This results in the smallness assumption on $\alpha_1$ and $\alpha_2$.
The key point of the proof is to carefully design $\omega$ to recover stability. Condition \eqref{e:localization-rate-3} comes from this construction.

The structure of the paper is then as follows. In Section \ref{s:single-front}, we revisit the existence and (in)stability of the steady states and single fronts. We also establish that Assumption \ref{a:speeds} can indeed be satisfied for any parameters $(d,r,\alpha_1,\alpha_2)$ and an open set of speeds $(c_1,c_2)$. In Section \ref{s:num-range}, we establish numerical range bounds for the time-dependent, weighted linear operator $\cL(t)$. Finally, Section \ref{s:proof-main-thm} provides the proof of the main Theorem \ref{t:main}.

\section*{Data availability statement}

The numerical simulations displayed in Figures \ref{fig:numerics1} and \ref{fig:numerics2} have been obtained using Matlab (Version 24.1.0, R2024a), and the code is available through the repository at \url{https://github.com/Bastian-Hilder/FrontCascade}.

\section{Endstates and single fronts}
\label{s:single-front}
Before proving Theorem \ref{t:main}, let us collect a few useful information about constant equilibria and single fronts.

\begin{lemma}
Under Assumption \ref{a:parameters}, $e_1$ is spectrally stable while $e_2$, $e_3$ and $e_4$ are spectrally unstable.
\end{lemma}
\begin{proof}
It is direct to compute that 
\begin{equation}
\label{e:expr-jacob}
J_g(u) = 
\begin{pmatrix}
r(1-2 u_1) & 0 \\
0 & 1-2 u_2 
\end{pmatrix}
+ 
\begin{pmatrix}
\alpha_1 u_2 & \alpha_1 u_1 \\
\alpha_2 u_2 & \alpha_2 u_1
\end{pmatrix}.
\end{equation}
Direct computations show that $J_g(e_1)$ has two stable eigenvalues, while $J_g(e_2)$, $J_g(e_3)$ and $J_g(e_4)$ all have at least one unstable eigenvalue. Using Fourier transform, spectral (in)stability of $J_g(e)$ is equivalent to $L^2(\RR)$ spectral (in)stability of $D\partial_{xx} + J_g(e)$.
\end{proof}

We now discuss the existence of invasion fronts. Insert the traveling wave ansatz \eqref{e:traveling-wave} into \eqref{e:rd} to obtain the profile equation 
\begin{equation}
\label{e:profile-equation}
0 = D\profile'' + c\profile' + g(\profile).
\end{equation}
\begin{lemma}
There exists $c_{1,*}>0$ such that for all $c_1 \geq c_{1,*}$, there exists $\profile_1:\RR \to \RR^2$ such that \eqref{e:profile-equation} is satisfied, and 
\begin{equation*}
\lim_{-\infty} \profile_1 = e_1, 
\hspace{4em}
\lim_{+\infty} \profile_1 = e_3.
\end{equation*}
Furthermore, components of $\profile_1$ are monotone, and there exists $x_1 \in \RR$ such that for all $x < x_1$
\begin{equation}
\label{e:phi-1-monotone}
\profile_{1,1}(x) \geq 1, 
\hspace{4em}
\profile_{1,2}(x) \geq 1, 
\end{equation} 
\end{lemma}
\begin{proof}
Following the definition from \cite[\S 2.2]{Volpert-Volpert-Volpert-94}, it is direct to see that system \eqref{e:rd} is monotone. In particular, notice that $e_1 > e_3$ component-wise. Applying \cite[Theorem 2.2]{Volpert-Volpert-Volpert-94} provides $c_{1,*}>0$ and profiles $\profile_1$ for all speeds $c_1 \geq c_{1,*}$. Both components of $\profile_1$ are monotone, thus decreasing, and \eqref{e:phi-1-monotone} follows from the component-wise inequality $e_1 > \transp{(1,1)}$.
\end{proof}

\begin{lemma}
There exists $c_{2,*}>0$ such that for all $c_2 \geq c_{2,*}$, there exists $\profile_2:\RR \to \RR^2$ such that \eqref{e:profile-equation} is satisfied, and 
\begin{equation*}
\lim_{-\infty} \profile_2 = e_3, 
\hspace{4em}
\lim_{+\infty} \profile_2 = e_4.
\end{equation*}
Furthermore, for every $\varepsilon>0$, there exists $x_2 \in \RR$ such that for all $x < x_2$,
\begin{equation}
\label{e:phi-2-monotone}
\profile_{2,1}(x) \geq 1 - \varepsilon. 
\end{equation} 
\end{lemma}
\begin{proof}
Looking for a profile 
$\profile_2 = 
\begin{pmatrix} 
\profile_{2,1} \\
\profile_{2,2}
\end{pmatrix}$
with $\profile_{2,2} = 0$, we see that the system \eqref{e:profile-equation} decouples, and is equivalent to the scalar KPP profile equation
\begin{equation*}
0 = d\profile_{2,1}'' + c_2 \profile_{2,1}' + r \profile_{2,1}(1 - \profile_{2,1}).
\end{equation*}
Applying \cite{Kolmogoroff-Pretrovsky-Piscounoff-37}, we obtain the slowest speed $c_{2,*} = 2\sqrt{dr}$, and existence of a decreasing profile connecting $1$ to $0$ for all speeds $c_2 \geq c_{2, *}$.
\end{proof}

We conclude this section with a convective stability criterion for the unstable equilibria.
\begin{lemma}
Let $c_1\geq c_{1,*}$, and $\kappa_1>0$ such that \eqref{e:localization-rate-1} holds. Then diagonal coefficients of $J_g(e_3) - c_1 \kappa_1 + D {\kappa_1}^2$ are negative.
\end{lemma}
\begin{proof}
Using expression \eqref{e:expr-jacob}, we find that diagonal coefficients of $J_g(e_3) - c_1 \kappa_1 + D {\kappa_1}^2$ are precisely the left-hand sides in \eqref{e:localization-rate-1}.
\end{proof}
\begin{lemma}
Let $c_2\geq c_{2,*}$, and $\kappa_2>0$ such that \eqref{e:localization-rate-2} holds. Then diagonal coefficients of $J_g(e_4) - c_2 \kappa_2 + D {\kappa_2}^2$ are negative.
\end{lemma}
\begin{proof}
The proof is almost identical to the previous one. Due to Assumption \ref{a:parameters}, diagonal coefficients of $J_g(e_4) - c_2 \kappa_2 + D {\kappa_2}^2$ are less than $r - c_2 \kappa_2 + d {\kappa_2}^2$. This is precisely the left-hand side in \eqref{e:localization-rate-2}.
\end{proof}

Finally, we show that it is indeed possible to satisfy all conditions in Assumption \eqref{a:speeds}. In particular, equilibria $e_3$ and $e_4$ are not remnantly unstable for large speeds, following \cite{Faye-Holzer-Scheel-Siemer-22} classification.
\begin{lemma}
\label{l:no-remnant-instability}
For any values of $(d,r,\alpha_1, \alpha_2)$, Assumption \ref{a:speeds} is fulfilled by taking $c_1$ and $c_2$ large enough.
\end{lemma}
\begin{proof}
We begin with a proof that \eqref{e:localization-rate-1} is fulfilled when $c_1$ is large enough. Given $c_1 > 2\sqrt{1 + \alpha_1}$ each condition in \eqref{e:localization-rate-1} have explicit solution sets
\begin{equation*}
\Absolute{\kappa_1 - \frac{c_1}{2}} < \frac{1}{2} \sqrt{{c_1}^2 - 4 (1 + \alpha_1)},
\end{equation*}
and
\begin{equation*}
\Absolute{\kappa_1 - \frac{c_1}{2d}} < \frac{1}{2d} \sqrt{{c_1}^2 + 4d r}.
\end{equation*}
Since these two intervals are respectively centered at $\frac{c_1}{2} > \frac{c_1}{2d}$, both conditions in \eqref{e:localization-rate-1} can be fulfilled simultaneously precisely when
\begin{equation*}
\frac{c_1}{2} - \frac{1}{2}\sqrt{{c_1}^2 - 4(1+\alpha)} < \frac{c_1}{2d} + \frac{1}{2d}\sqrt{{c_1}^2 + 4dr}.
\end{equation*}
Dividing by $\frac{c_1}{2}$ and Taylor expanding when $c_1\to +\infty$, this condition becomes
\begin{equation*}
\frac{1 + \alpha - r}{{c_1}^2} + \cO_{c_1 \to +\infty}\left(\frac{1}{{c_1}^4}\right) < \frac{1}{d},
\end{equation*}
which is satisfied for all large enough $c_1$. 

We now turn to \eqref{e:localization-rate-2}. The set of admissible $\kappa_2$ for this inequality is non-empty as soon as $c_2 > 2\sqrt{dr}$. Furthermore, when $c_2 \to +\infty$ the lower and upper bounds expand as 
\begin{equation}
\label{e:bound-kappa-2}
\frac{r}{c_2} + \cO_{c_2 \to +\infty}\left(\frac{1}{{c_2}^3}\right) < \kappa_2 < \frac{c_2}{d} - \frac{r}{c_2} + \cO_{c_2 \to +\infty}\left(\frac{1}{{c_2}^3}\right).
\end{equation}

To conclude, let us discuss \eqref{e:localization-rate-3}. We see from \eqref{e:bound-kappa-2} that $\kappa_2 > \kappa_1$ is always possible when $c_2$ is large enough. Increasing $c_2$ if necessary then ensures that \eqref{e:localization-rate-3} holds. The proof is complete.
\end{proof}

\section{Numerical range bounds}
\label{s:num-range}
The goal of this section is to prove a time-uniform resolvent bound on $\cL(t)$, see the later Proposition \ref{p:resolvent-bound}. It essentially reduces to locate the numerical range of $\cL(t)$, a subset of the complex plane defined as 
\begin{equation*}
R_{\cL(t)} \deq \Set{\scalp{\cL(t) u, u}: u \in H^2(\RR, \CC^2), \norm{u}_{L^2(\RR)} = 1}.
\end{equation*}
Since $\cL(t)$ is parabolic, our goal is to include $R_{\cL(t)}$ in a stable sector. To keep the presentation simple, we first showcase the estimate for a single scalar front, and then present the two-stage invasion case. 

The simpler scalar case discussion is independent of Theorem \ref{t:main} proof. While such a wave is usually handled with a co-moving frame approach to get rid of time dependence, we take the occasion to illustrate that our method also applies in the steady frame.

\subsection{Single scalar wave}
\label{s:num-range-one-front}
In this subsection only, we replace $\uu(t,x)$ with a single scalar KPP front $p(x - ct)$ with speed $c > 2\sqrt{rd}$. We further write the weight in exponential form: $\omega(t,x) = \re^{\phi(t,x)}$. The expression of $\cL(t)$ then reduces to
\begin{equation*}
\cL^{\kpp}(t) = d \partial_{xx} + 2d\phi_x \partial_x + r(1-2\profile(x-ct)) - \phi_t + d (\phi_{xx} + {\phi_x}^2).
\end{equation*}
From the standard identity $\scalp{\phi_x u_x, u} = -\frac{1}{2} \scalp{\phi_{xx} u, u}$, that is obtained with one integration by parts, we compute that 
\begin{align*}
\Real{\scalp{\cL^\kpp u, u}} & {} = - d \norm{u_x}_{L^2}^2 + \scalp{a_0 u, u},\\
\Imag{\scalp{\cL^\kpp u, u}} & {} = 2d \Imag{\Scalp{\phi_x u_x, u}},
\end{align*}
where 
\begin{equation*}
a_0 \deq r (1 - 2\profile) - \phi_t + d {\phi_x}^2.
\end{equation*}
\begin{lemma}
There exist a positive constant $\eta$ and a $\cC^\infty$ function $\phi :(0,+\infty)\times \RR$ such that for all $(t,x)\in(0,+\infty)\times\RR$
\begin{equation*}
a_0(t,x) \leq -\eta.
\end{equation*}
\end{lemma}
\begin{proof}
We can always assume that $\frac34 \leq \profile(x-ct) \leq 1$ when $x-ct \leq 1$. Indeed, $\profile$ is monotonic, and the profile equation is translation invariant. Thus $\profile$ can be replaced with $\profile(\cdot - x_0)$ for $x_0 \in \RR$. From the condition on $c$, there exists an $\eta>0$ such that
\begin{equation*}
d \kappa^2 - c\kappa + r = -\eta
\end{equation*}
admits positive solutions when solved for $\kappa$. Let $\kappa$ be one of these solutions and define
\begin{equation*}
\phi(t,x) = 
\begin{cases}
0 & \text{ if } x-ct \leq -1,\\
-\frac{\kappa}{4}(x-ct+1)^2 & \text{ if } x-ct \in (-1,1),\\
-\kappa (x-ct) & \text{ if } x-ct \geq 1.
\end{cases}
\end{equation*}
It is direct to check that $\phi$ is $\cC^1$ with respect to its variables.

When $x-ct\geq 1$, the choice of $\kappa$ ensures that 
\begin{equation*}
a_0(t,x) = -2r\profile(x-ct) - \eta \leq -\eta.
\end{equation*}
When $x-ct \in (-1,1)$, we compute that
\begin{equation*}
a_0(t,x) = -2r\profile(x-ct) + P(x-ct),
\end{equation*}
where $P$ is the convex second-order polynomial
\begin{equation*}
P(y) \deq \left(r - \frac{c\kappa}{2} + \frac{d\kappa^2}{4}\right) + \left(\frac{d\kappa^2}{2} - \frac{c\kappa}{2}\right) y + \frac{d\kappa^2}{4} y^2.
\end{equation*}
In particular for all $y\in(-1,1)$,
\begin{equation*}
P(y) \leq \max(P(-1), P(1)) = \max(r, -\eta) = r.
\end{equation*}
Thus when $x-ct\in(-1,1)$, 
\begin{equation*}
a_0(t,x) \leq r(1-2\profile(x-ct)) \leq -\frac{r}2.
\end{equation*}
We remark that the last chain of inequality is still valid when $x-ct \leq -1$. Shrinking $\eta$ if necessary, we proved the claimed bound for all $(t,x)\in (0,+\infty) \times\RR$. 

To conclude the proof, we explain how to improve regularity from $\cC^1$ up to $\cC^\infty$. We set  a partition of unity $\theta:\RR\to [0,1]$ such that 
\begin{equation*}
\theta(y) = 
\begin{cases}
0 & \absolute{y} > 2,\\
1 & \absolute{y} < \frac{3}{2}.
\end{cases}
\end{equation*}
First, notice that the restriction $\phi_{|[ct - 2, ct + 2]}$ is time-independent after a suitable space translation. We approximate it using density of smooth functions: For any $\delta > 0$, there exists a $\cC^\infty$ map $\tilde{\phi} : [- 2, 2] \to \RR$ such that $\norm{\phi(\cdot + ct) - \tilde{\phi}}_{W^{1,\infty}(- 2,2)} \leq \delta$. The convex combination 
\begin{equation*}
(t, x) \mapsto \theta(x - ct) \tilde{\phi}(x - ct) + (1 - \theta(x - ct)) \phi(t,x)
\end{equation*}
is $\delta$ close to $\phi$ in $W^{1,\infty}$-norm, so that the bound $a_0 \leq -\eta$ is still valid upon shrinking $\eta$.
\end{proof}

Assuming that $\norm{u}_{L^2} = 1$, the previous lemma ensures that 
\begin{equation*}
\Real{\scalp{\cL^\kpp u, u}} \leq -\norm{u_x}_{L^2}^2 - \eta.
\end{equation*}
Turning to the imaginary part, we obtain
\begin{equation*}
\absolute{\Imag{\scalp{\cL^\kpp u, u}}} \leq 2d\kappa \norm{u_x}_{L^2}.
\end{equation*}
Thus, for all $u\in H^2(\RR)$ with $\norm{u}_{L^2} = 1$, there exists $\xi \geq 0$ such that 
\begin{align*}
\scalp{\cL^\kpp u, u} \in {} & \Set{\lambda \in \CC : \Real{\lambda} \leq - \xi^2 - \eta \ \  \text{ and } \ \ \absolute{\Imag{\lambda}} \leq 2d\kappa \xi},\\
\subset {} & \Set{\lambda \in \CC : \Real{\lambda} \leq - C(1 + \absolute{\Imag{\lambda}}^2)}
\end{align*}
for some constant $C>0$.

\subsection{Two front superposition}
Coming back to the actual notations, we proceed similarly as in the previous subsection. To lighten notations, we may again write the weight in exponential form:
\begin{equation}
\label{e:def-phi}
\omega(t,x) \rdeq \re^{\phi(t,x)}.
\end{equation}
\begin{lemma}
\label{l:real-imag-scalp}
Let $\phi$ as in \eqref{e:def-phi}. For each $t\geq 0$ and each $u \in H^2(\RR,\CC^2)$, 
\begin{align*}
\Real{\scalp{\cL(t) u, u}} &{} = -\scalp{D u_x, u_x} + \Real{\Scalp{A_0 u, u}},\\
\Imag{\scalp{\cL(t) u, u}} &{} = 2\Imag{\scalp{D \phi_x u_x, u}} + \Imag{\Scalp{J_g(\uu) u, u}},
\end{align*}
where
\begin{equation*}
A_0 \deq J_g(\uu) - \phi_t + D(\phi_x)^2.
\end{equation*}
\end{lemma}
\begin{proof}
Since the coefficients of $\cL$ are real-valued, and most of them are diagonal, the expression of $\Imag{\scalp{\cL u, u}}$ is as claimed. The expression for the real part follows from the identity $\Real{\scalp{\phi u_x, u}} = -\frac{1}{2}\scalp{\phi_{xx} u, u}$.
\end{proof}

The following lemma ensures that the system case reduces to the scalar case when coupling coefficients are small enough.
\begin{lemma}
\label{l:finite-dim-numerical-range}
Let $a$, $b$, $c$, $d$ in $\RR$ such that $a<0$, $d<0$ and $\frac{b + c}{2} < \sqrt{ad}$. Let
\begin{equation*}
A \deq 
\begin{pmatrix}
a & b \\
c & d
\end{pmatrix}.
\end{equation*}
Then there exists $\eta > 0$ such that for any $Z = \transp{(z_1,z_2)} \in \CC^2$
\begin{equation*}
\Real{\Scalp{A Z, Z}} \leq -\eta \absolute{Z}^2.
\end{equation*}
\end{lemma}
\begin{proof}
We compute that 
\begin{equation*}
\Real{\Scalp{AZ, Z}} = a \absolute{z_1}^2 + d \absolute{z_2}^2 + (b + c)(\Real{z_1}\Real{z_2} + \Imag{z_1}\Imag{z_2}).
\end{equation*}
Using Young's inequality $xy \leq \frac{\varepsilon}{2} x^2 + \frac{1}{2\varepsilon} y^2$, to bound the right-hand side, we obtain the estimate
\begin{equation*}
\Real{\Scalp{AZ, Z}} \leq \left(a + \varepsilon\frac{b + c}{2}\right) \absolute{z_1}^2 + \left(d + \frac{b + c}{2\varepsilon}\right) \absolute{z_2}^2.
\end{equation*}
Setting $\varepsilon = \sqrt{\frac{a}{d}}$, it rewrites as
\begin{equation*}
\Real{\Scalp{AZ, Z}} \leq \left(\frac{b + c}{2} - \sqrt{ad}\right)\left(\varepsilon \absolute{z_1}^2 + \frac{1}{\varepsilon} \absolute{z_2}^2\right),
\end{equation*}
which completes the proof.
\end{proof}

\begin{lemma}
\label{l:diagonal-coeff}
There exist positive constants $\eta$, $\alpha$ and $\phi: (0,+\infty) \times \RR \to (-\infty, 0]$ such that for all $\alpha_1, \alpha_2 \in (0, \alpha)$ and $(t,x)\in (0,+\infty) \times \RR$, the diagonal coefficients of $A_0$ are less than $-\eta$.
\end{lemma}
\begin{proof}
To keep formulas readable, we drop the initial shifts in this proof. That is we assume $\psi_1 = \psi_2 = 0$. The proof adapts to non-zero shifts by replacing each occurrence of $x - c_j t$ with $x - c_j t - \psi_j$.

We decompose the half plane into five regions: $(0,+\infty) \times \RR = I_1 \cup ... \cup I_5$. The first three correspond to $\uu$ being close to constants:
\begin{align*}
I_1 = {} & \Set{(t,x) : x - c_1 t \leq - 1},\\
I_3 = {} & \Set{(t,x) : x - c_1 t \geq 1 \text{ and } x - c_2 t \leq -1},\\
I_5 = {} & \Set{(t,x) : x - c_2 t \geq 1}.
\end{align*}
The remaining two correspond to transition regions of $\uu$:
\begin{align*}
I_2 = {} & \Set{(t,x) : x - c_1 t \in (-1,1)},\\
I_4 = {} & \Set{(t,x) : x - c_2 t \in (-1,1)}.
\end{align*}
Let us collect some useful bounds on the profile $\uu$. Upon taking a smaller $\alpha$, we can assume that 
\begin{equation}
\label{e:bound-profile-1}
\left(\uu(t,x)\right)_1, \left(\uu(t,x)\right)_2 \leq 2.
\end{equation}
As for a single front, we can always translate both $\profile_1$ and $\profile_2$. More precisely, let $\varepsilon>0$ be chosen in a few lines. Using \eqref{e:phi-1-monotone}-\eqref{e:phi-2-monotone}, we can assume that 
\begin{align}
\stepcounter{equation}
\label{e:bound-profile-2}
\tag{\theequation.\textit{a}}
(t,x) \in I_1\cup I_2 &
\hspace{2em}
\implies 
\hspace{2em}
\left(\uu(t,x)\right)_1, \left(\uu(t,x)\right)_2 \geq 1,\\
\label{e:bound-profile-3}
\tag{\theequation.\textit{b}}
(t,x) \in I_3 \cup I_4 &
\hspace{2em}
\implies 
\hspace{2em}
\left(\uu(t,x)\right)_1 \geq 1-\varepsilon,\qquad \left(\uu(t,x)\right)_2 \geq 0.
\end{align} 

In the region where $\uu$ is close to unstable equilibria, we need exponential decay. Let $\kappa_1$ and $\kappa_2$ which satisfy Assumption \ref{a:speeds}. From \eqref{e:localization-rate-2}, we can choose $\varepsilon$ so small that 
\begin{equation}
\label{e:stability-cond-I3}
d {\kappa_1}^2 - c_1 \kappa_1 - r + 2 r\varepsilon < 0.
\end{equation}
We then define
\begin{equation}
\label{e:weight-def-1}
\phi(t,x) =
\begin{cases}
0 &  (t,x) \in I_1, \\[1ex]
-\kappa_1 (x - c_1 t) & (t,x) \in I_3, \\[1ex]
-\kappa_1 (c_2 - c_1) t - \kappa_2 (x - c_2 t) & (t,x) \in I_5.
\end{cases}
\end{equation}
On the transition regions, we further impose
\begin{equation}
\label{e:weight-def-2}
\phi(t,x) =
\begin{cases}
-\frac{\kappa_1}{4} (x - c_1 t + 1)^2 & (t,x) \in I_2,\\[1ex]
-\kappa_1 (x-c_1 t) - \frac{\kappa_2-\kappa_1}{4} (x - c_2 t + 1)^2 & (t,x) \in I_4.
\end{cases}
\end{equation}
It is direct to check that $\phi$ is $\cC^1$ with respect to its variables. From \eqref{e:weight-def-1}, we compute:
\begin{equation*}
A_0 = 
\begin{cases}
J_g(\uu) & (t,x)\in I_1,\\[1ex]
J_g(\uu) - \kappa_1 c_1 + D{\kappa_1}^2 & (t,x)\in I_3,\\[1ex]
J_g(\uu) - \kappa_2 c_2 + D{\kappa_2}^2 + \kappa_1(c_2 - c_1) & (t,x)\in I_5,
\end{cases}
\end{equation*}
while \eqref{e:weight-def-2} leads to:
\begin{equation*}
A_0 = 
\begin{cases}
J_g(\uu) + P_2(x-c_1 t) & (t,x)\in I_2,\\[1ex]
J_g(\uu) + P_4(x-c_2 t) & (t,x)\in I_4,
\end{cases}
\end{equation*}
with second-order polynomials
\begin{align*}
P_2(y) = {} & \frac{{\kappa_1}^2}{4} D - \frac{c_1 \kappa_1}{2} + \frac{y}{2} \left({\kappa_1}^2 D - c_1 \kappa_1\right) + \frac{y^2}{4}{\kappa_1}^2 D,\\
P_4(y) = {} & \frac{(\kappa_2 + \kappa_1)^2}{4} D - \frac{c_2 (\kappa_2 - \kappa_1)}{2} - \kappa_1 c_1 \\
& {} + \frac{y}{2} \left(({\kappa_2}^2 - {\kappa_1}^2) D - c_2(\kappa_2 - \kappa_1)\right) + \frac{y^2}{4}(\kappa_2 - \kappa_1)^2 D.
\end{align*}
We now successively bound diagonal coefficients of $A_0$ in the different regions.
\begin{itemize}
\item $I_1:$ Using the expression \eqref{e:expr-jacob} for $J_g(u)$ and \eqref{e:bound-profile-2}, we get that diagonal coefficients are less than $-r + 2\alpha_1$ and $-1 + 2\alpha_2$ respectively. Both are negative when $\alpha$ is small.
\item $I_2:$ We first compute $P_2$ on the boundary:
\begin{equation*}
P_2(-1)  = 0,
\hspace{4em}
P_2(1) = D{\kappa_1}^2 - c_1 \kappa_1.
\end{equation*}
Then, we handle $P_2$ using convexity of its coefficients, and $J_g(\uu)$ using the same bound as in $I_1$ to obtain that diagonal coefficients of $A_0$ are less than (respectively)
\begin{align*}
& 2\alpha + \max(-r, d {\kappa_1}^2 - c_1 \kappa_1 - r),\\
& 2\alpha + \max(-1, {\kappa_1}^2 - c_1 \kappa_1 - 1).
\end{align*}
Using Assumption \ref{a:speeds} and taking $\alpha$ small enough, both are negative.
\item $I_3:$ From \eqref{e:bound-profile-3}, diagonal coefficients are less than  (respectively)
\begin{align*}
& 2\alpha + d {\kappa_1}^2 - c_1 \kappa_1 - r + 2\varepsilon r, \\
& 2\alpha + {\kappa_1}^2 - c_1 \kappa_1 - 1.
\end{align*}
Using \eqref{e:stability-cond-I3} and taking $\alpha$ small enough, those are negative.
\item $I_4:$ We compute $P_4$ on the boundary: 
\begin{equation*}
P_4(-1) = {\kappa_1}^2 D - c_1 \kappa_1,
\hspace{4em}
P_4(1) = {\kappa_2}^2 D - c_2 \kappa_2 + \kappa_1(c_2 - c_1).
\end{equation*}
Using \eqref{e:bound-profile-3} to control $J_g(\uu)$, and convexity of $P_4$, we get that diagonal coefficients of $A_0$ are less than (respectively)
\begin{align*}
2(r\varepsilon + \alpha) + {} & \max\bigg(d {\kappa_1}^2 - c_1 \kappa_1 - r, \quad d {\kappa_2}^2 - c_2 \kappa_2 - r + \kappa_1 (c_2 - c_1)\bigg),\\
& \max\bigg({\kappa_1}^2 - c_1 \kappa_1 + 1, \quad {\kappa_2}^2 - c_2 \kappa_2 + 1 + \kappa_1 (c_2 - c_1)\bigg).
\end{align*}
Using \eqref{e:stability-cond-I3} and \eqref{e:localization-rate-3}, the first coefficient is negative. The second one follows using \eqref{e:localization-rate-1}, \eqref{e:localization-rate-3}, and $r > 1$.
\item $I_5:$ Using Assumption \ref{a:parameters} and $(\uu)_1, (\uu)_2 \geq 0$, we obtain that diagonal coefficients of $A_0$ are less than
\begin{equation*}
d {\kappa_2}^2 - c_2 \kappa_2 + r + \kappa_1 (c_2 - c_1).
\end{equation*}
Equation \eqref{e:localization-rate-3} precisely ensures that this quantity is negative.
\end{itemize}
To conclude the proof, we can improve the regularity of $\phi$ in a very similar way as for the single front case. We notice that close to a regularity defect, $\phi$ only depends on $x-c_1 t$ or $x-c_2 t$, and use the approximation by smooth functions twice.
\end{proof}

\begin{proposition}
\label{p:resolvent-bound}
Fix $d > 1, r > 1$ and let Assumptions \ref{a:parameters}, \ref{a:ordered-fronts}, and \ref{a:speeds} be satisfied. Then, there exist positive constants $\eta$ and $\alpha$ such that for all $\alpha_1, \alpha_2 \in (0,\alpha)$ and $t\geq 0$, the spectrum of $\cL(t)$ is included in the sector 
\begin{equation}
\label{e:def-S}
S \deq \Set{\lambda \in \CC : \Real{\lambda} \leq -\eta\left(1 +  \absolute{\Imag{\lambda}}\right)}.
\end{equation}
Furthermore, for each $t\geq 0$ the following resolvent estimate holds. For each $\lambda \notin S$ and each $f\in L^2(\RR)$,
\begin{equation}
\label{e:res-estimate}
\Norm{(\lambda - \cL(t))^{-1} f}_{L^2(\RR)} \leq \frac{\Norm{f}_{L^2(\RR)}}{\mathrm{dist}(\lambda, S)}.
\end{equation}
\end{proposition}
\begin{proof}
We show that the numerical range of $\cL(t)$ is included in $S$. It implies the claim, as shown by \cite[Lemma 4.1.9]{Kapitula-Promislow-13}.

Let $u \in H^2(\RR,\CC^2)$ such that $\norm{u}_{L^2} = 1$. Using Lemma \ref{l:real-imag-scalp}, we can bound the real and imaginary part of $\scalp{\cL u, u}$. Since $\phi_x$ and the coefficients of $J_g(\uu)$ are bounded,
\begin{align}
\nonumber
\Absolute{\Imag{\Scalp{\cL u, u}}} &{} \leq C \Norm{u_x}_{L^2} \Norm{u}_{L^2} + C\norm{u}_{L^2}^2, \\
\label{e:bound-imag}
&{} \leq C \left(1 + \Norm{u_x}_{L^2}\right).
\end{align}
To estimate the real part, we rely on Lemma \ref{l:finite-dim-numerical-range}. From Lemma \ref{l:diagonal-coeff}, diagonal coefficients are negative. Off-diagonal coefficients read $\alpha_1 \uu_1$ and $\alpha_2 \uu_2$, thus Lemma \ref{l:finite-dim-numerical-range} applies when $\alpha$ is small enough. This yields a $C>0$ such that 
\begin{equation}
\label{e:bound-real}
\Real{\Scalp{\cL u, u}} \leq - C (1 + \norm{u_x}_{L^2}^2).
\end{equation}
Combining \eqref{e:bound-imag} and \eqref{e:bound-real}, we see that there exists $\xi\in \RR$ such that $\scalp{\cL u, u}$ belongs to the set
\begin{equation*}
\Set{\lambda\in \CC: \quad \Real{\lambda} \leq -C(1 + \xi^2) \quad \text{and} \quad
\Absolute{\Imag{\lambda}} \leq C(1 + \absolute{\xi})},
\end{equation*}
which is itself a subset of \eqref{e:def-S} for a small enough $\eta$.
\end{proof}

\section{Proof of the main result}\label{s:proof-main-thm}
\begin{proof}[Proof: Theorem \ref{t:main}]
To control solutions to \eqref{e:weighted-dynamic}, we rely on the evolution system theory \cite[\S 5.2, \S 5.6]{Pazy-83}. The family of operators $\cL(t) : H^2(\RR) \subset L^2(\RR) \to L^2(\RR)$ is stable with constant $(1,-\eta)$, due to the resolvent bound \eqref{e:res-estimate}. It is direct to check that $\tilde{\cL}(t) \deq \cL(t) + \eta$ satisfies 
\begin{equation*}
\Norm{\left(\tilde{\cL}(t_1) - \tilde{\cL}(t_2)\right) \tilde{\cL}(t_3)^{-1} f}_{L^2}\leq C \absolute{t_1 - t_2} \norm{f}_{L^2},
\end{equation*}
since $\uu$ and $\omega$ are continuous. Applying \cite[Theoreme 6.1]{Pazy-83} to $\tilde{\cL}$, we recover that $z(t) = e^{\eta t} w(t)$ satisfies
\begin{equation*}
\norm{z(t)}_{L^2} \leq C \norm{z(0)}_{L^2}.
\end{equation*}
Unfolding the change of variable concludes the proof.
\end{proof}

\bibliographystyle{alphaabbr}
\bibliography{terraces}

\end{document}